\documentclass[11pt]{article}

\usepackage{fullpage}
\usepackage{amsmath,amsthm,amsfonts,dsfont}
\usepackage{amssymb,latexsym,graphicx}
\usepackage{palatino}
\usepackage{mathpazo}
\usepackage{stmaryrd}
\usepackage{mathtools}
\usepackage{hyperref}

\newtheorem{theorem}{Theorem}[section]

\newtheorem{lemma}[theorem]{Lemma}

\newtheorem{claim}[theorem]{Claim}

\newtheorem{corollary}[theorem]{Corollary}

\renewcommand{\epsilon}{\varepsilon}

\makeatletter
\def\imod#1{\allowbreak\mkern8mu({\operator@font mod}\,\,#1)}
\makeatother

\DeclareRobustCommand{\stirling}{\genfrac\{\}{0pt}{}}

\newcommand{\marginlabel}[1]%
{\mbox{}\marginpar{\it{\raggedleft\hspace{0pt}#1}}}

\renewcommand{\epsilon}{\varepsilon}

\DeclareMathOperator*{\E}{E}

\DeclareMathOperator{\diag}{diag}

\newcommand{\N}{\ensuremath{\mathbb{N}}}

\newcommand{\R}{\ensuremath{\mathbb{R}}}

\makeatletter
\def\imod#1{\allowbreak\mkern10mu({\operator@font mod}\,\,#1)}
\makeatother

\def\diag{\mathrm{diag}}

\newif\ifnewprog\newprogfalse

\ifnewprog

\newlength{\pgmtab}  
\fi

\newcounter{lecnum}

\newlength{\tpush}
\DeclarePairedDelimiter\parens{(}{)}
\DeclarePairedDelimiter\bracks{[}{]}
\DeclarePairedDelimiter\braces{\{}{\}}

\newif\ifnotes\notesfalse

\ifnotes
\usepackage{color}
\definecolor{mygrey}{gray}{0.50}
\newcommand{\notename}[2]{{\textcolor{mygrey}{\footnotesize{\bf (#1:} {#2}{\bf ) }}}}
\newcommand{\noteswarning}{{\begin{center} {\Large WARNING: NOTES ON}\end{center}}}

\else

\newcommand{\notename}[2]{{}}
\newcommand{\noteswarning}{{}}

\fi

\begin{document}

\title{A Sharp Tail Bound for the Expander Random Sampler}

\author{
Shravas Rao\thanks{Courant Institute of Mathematical Sciences, New York
 University.}
~\thanks{
    This material is based upon work supported by the National Science Foundation Graduate Research Fellowship Program under Grant No. DGE-1342536.
}
\and
Oded Regev\footnotemark[1]
~\thanks{Supported by the Simons Collaboration on Algorithms and Geometry and by the National Science Foundation (NSF) under Grant No.~CCF-1320188. Any opinions, findings, and conclusions or recommendations expressed in this material are those of the authors and do not necessarily reflect the views of the NSF.}\\
}
\date{}
\maketitle

\noteswarning

%

\begin{abstract}
Consider an expander graph in which a $\mu$ fraction of the vertices are marked. 
A random walk starts at a uniform vertex and at each step continues to a random neighbor. 
Gillman showed in 1993 that the number of marked vertices seen in a random walk of length $n$ is 
concentrated around its expectation, $\Phi := \mu n$, independent of the size of the graph. 
Here we provide a new and sharp tail bound, improving on the existing bounds whenever $\mu$ is not too large. 
\end{abstract}

\section{Introduction}

Let $G = (V, E)$ be a finite regular undirected graph, and let $A$ be its normalized adjacency matrix.  
Let 
\[
  \lambda(G) = \|A-J\|
\] 
be the second largest absolute value of an eigenvalue of $A$,
where $\| \cdot \|$ denotes
the operator norm, and $J$ is the matrix whose entries are all $1/|V|$. 
There exist families of $d$-regular graphs $G$ 
for some constant $d$ so 
that $\lambda(G)$ is bounded above by a
 constant less than $1$~\cite{LPS88, M88, F08}.
Such graphs are known as \emph{expander graphs}.

An \emph{expander sampler} samples vertices of an expander graph 
by performing a simple random walk on the graph.
Note that when $G$ is the complete graph with self-loops,
we have $A = J$ and $\lambda(G) = 0$,
and this corresponds to sampling vertices uniformly and independently.
One remarkable property of expander samplers is that the sampled vertices behave in 
various ways like vertices chosen uniformly and independently.
This is even though the sampled vertices are not independent---knowing
one vertex in the random walk narrows down the number of choices of the next
vertex in the walk to just $d$.
More precisely, fix a graph $G$, an arbitrary $n \ge 1$, and functions $f_1, \ldots, f_n: V \rightarrow [0, 1]$ (where often $f_1=\cdots=f_n=f$). 
Let $(Y_1, \ldots, Y_n)$ be the simple
 random walk on $G$ of length $n$, i.e., $Y_1$ is chosen uniformly at random from $V$, 
and each subsequent $Y_i$ is chosen uniformly
 from the set of neighbors of $Y_{i-1}$.  
We will be interested in the random variable
$S_n = f_1(Y_1)+\cdots+f_n(Y_n)$.

When the vertices are sampled uniformly and independently, the
behavior of $S_n$ is described
by the classical central limit theorem.  In particular, the distribution
of $S_n$ tends towards a normal distribution.
Refinements include the Berry-Esseen theorem, which describes the rate at which
this occurs~\cite{B41, E42}, and various large deviation bounds.
In all cases, analogous results hold for expander samplers, 
with the appropriate dependence on $\lambda(G)$. 
For a more complete
listing of these results, see Section 1.3.2 of the thesis of Mann~\cite{M96}.

Here we are interested in bounding $\E[\alpha^ {S_n}]$, the moment generating
function of $S_n$. 
Such bounds easily imply tail bounds on $S_n$ by Markov's inequality. 
When vertices are sampled uniformly and independently, letting $\mu_i = \E[f_i(v)]$ and $\Phi := \E[S_n] = \mu_1 + \cdots + \mu_n$,
 a simple calculation shows that
the moment generating function is bounded above by
\[
\E[\alpha^{ S_n}]
\leq
\exp((\alpha -1) \cdot \Phi ) \; .
\]

%

In this note we prove an analogous bound for the case of expander walks.
\begin{theorem}\label{thm:maindiff}
Let $G = (V, E)$ be a regular undirected graph and let $\lambda = \lambda(G)$.
Let $f_1, \ldots, f_n: V \rightarrow [0, 1]$, and let $\mu_i = \E[f_i(v)]$.
For a random walk 
$(Y_1, \ldots, Y_n)$, let $S_n = f_1(Y_1)+\cdots+f_n(Y_n)$
and $\Phi := \E[S_n] = \mu_1 + \cdots + \mu_n$.
Then for $1 < \alpha < 1/\lambda$, 
\begin{align}\label{eq:maindiffbound}
\E[\alpha^{S_n}]
\leq
\exp\parens[\bigg]{(\alpha-1) \cdot \Phi \cdot \parens[\bigg]{\frac{1-\lambda}{1-\alpha\lambda}}}.
\end{align}
\end{theorem}

As stated previously, bounds on the moment generating function immediately imply tail bounds.
In particular, by plugging in $\alpha = \lambda^{-1} - (1-\lambda)/(t^{1/2} \lambda^{3/2})$
in Theorem~\ref{thm:maindiff} we obtain the following.
\begin{corollary}\label{cor:tailbound}
In the setting of Theorem~\ref{thm:maindiff}, for all $t > 1/\lambda$,
\begin{equation}\label{eq:tailbound}
\Pr[S_n \geq t \Phi]
\leq
\parens[\Big]{\frac{1}{\lambda} - \frac{1-\lambda}{t^{1/2} \lambda^{3/2}}}^{-t \Phi} \exp(\Phi \cdot (1-\lambda)(\sqrt{t \lambda}-1)/\lambda) \; .
\end{equation}
\end{corollary}
For instance, for $\lambda = 1/2$, we obtain
\[
\Pr[S_n \geq t \Phi]
\leq
\parens[\Big]{2 - \sqrt{\frac{2}{t}}}^{-t \Phi} \exp(\Phi \cdot (\sqrt{t/2}-1)) \; .
\]
We also note that for large $t$, the bound in~\eqref{eq:tailbound} is roughly $\lambda^{t \Phi}$,
which is again close to tight by the example in Section~\ref{sec:tightness}.

\paragraph{Related Work.}
Theorem~\ref{thm:maindiff} and Corollary~\ref{cor:tailbound} are 
part of a long line of work on proving tail bounds
for expander walks starting from Gillman~\cite{Gillman1993,G98},
and further developed in~\cite{D95, K97, L98, Wagner08, LCP04, H08, ChungLLM12, HazlaH15, P15, NaorRR17}. 

Closest to our work are the results of Lezaud~\cite{L98} 
and Le{\'o}n and Perron~\cite{LCP04}.
Lezaud proved a bound of a form similar to~\eqref{eq:maindiffbound} (see
his Remark 2) with a large constant in the exponent.
Le{\'o}n and Perron~\cite{LCP04} later improved on his result,
proving a sharp bound on the moment generating function with
corresponding tail bounds. 
Although their results are sharp, their bounds appear to be 
somewhat unwieldy, and they do not explicitly include 
convenient tail bounds as in our Corollary~\ref{cor:tailbound}.
Another difference is that both results~\cite{L98, LCP04}
assume $f_1 = \cdots = f_n$, although with some work, 
one can probably extend them 
to the case of general $f_1,\ldots,f_n$ as we consider here 
(and which is required for the application mentioned below). 
A final difference between our work and past work is in terms of proof techniques. 
Most work in this area including~\cite{L98, LCP04} use perturbation theory,
while some more recent work uses direct linear algebra arguments. 
Our proof is of the latter type, and arose from our joint work with Naor~\cite{NaorRR17}.
We also feel that our proof is somewhat cleaner than some previous proofs. 

\subsection{Application to low-randomness samplers}

One motivation for our work comes from a recent paper of Meka~\cite{M15},
as explained next. 
For a set $V$, define a \emph{sampler} from $V$ of length $n$ as a function whose range 
is $V^n$. The \emph{seed length} of a sampler is defined as the logarithm of the cardinality of the function's domain,
and can be seen as the number of random bits necessary to sample from the function. 
With this terminology, an expander sampler for a $d$-regular graph $G = (V, E)$ 
has seed length $\log_2 |V| + (n-1) \log_2 d$.
In particular, if $G$ is a constant degree expander, the seed length of $H$ is $\log_2 |V| + O(n)$.

In~\cite{M15}, Meka constructs a (non-expander) sampler satisfying a bound similar to
that in Eq.~\eqref{eq:maindiffbound} using the worse seed length of
$O(n + \log |V| + n(\log \log |V| + \log(1/(\mu_1+\cdots+\mu_n)))/\log n)$.
Our results show that the sampler of~\cite{M15} can be replaced with an
expander sampler, providing an improved seed length and a simpler construction. 
Notice however that Meka does not require optimal constants, and instead of
using our bounds, one could also use some of the earlier work, such as~\cite{L98,Wagner08}
(after extending them to deal with $f_1,\ldots,f_n$ that are not necessarily all equal;
see also~\cite[Section 3.2]{ChungLLM12}).

\subsection{Sharpness}\label{sec:tightness}

We now sketch an argument showing that Theorem~\ref{thm:maindiff} is sharp in the following sense. 
Fix arbitrary $\lambda \in [0,1)$ and $\Phi > 0$, and consider the matrix 
$A = \lambda I + (1-\lambda) J$ of dimensions $|V| \times |V|$. 
This corresponds to the walk where at each step 
we either stay in place with probability $\lambda$, or choose a uniform vertex with
probability $1-\lambda$. (Strictly speaking, $A$ does not correspond to a regular
unweighted graph; it is straightforward to modify the example to this case.)
Let $f_1=\cdots=f_n=f$ be the function that assigns $1$ to a $\mu=\Phi/n$ fraction
of ``marked'' vertices and $0$ to the remaining vertices (where we assume for simplicity
that $\Phi |V|/n$ is integer). Equivalently, one can consider a Markov chain with two states, 
one marked and one unmarked; a step in the chain stays in the same state with probability $\lambda$ 
and otherwise chooses from the stationary distribution, which assigns mass $\mu$ to the marked state 
and $1-\mu$ to the unmarked state. 

Then we claim that as $n$ goes to infinity,
the left-hand side of Eq.~\eqref{eq:maindiffbound} converges to the right-hand side. 
To see that, we say that a step of the walk is a ``hit'' if 
(1) the walk chooses a uniform vertex (which happens with probability $1-\lambda$), and (2) that chosen vertex is marked.
Then observe that the random variable counting the number of hits during the walk 
converges to a Poisson distribution with 
expectation $(1-\lambda)\Phi$ (since it is the sum of $n$ independent Bernoulli random variables,
each with probability $(1-\lambda)\mu$ of being $1$). Moreover, each time a hit occurs,
 we stay in that vertex a number of steps that is distributed like 
a geometric distribution with success probability $1-\lambda$. (We are ignoring 
here lower order effects, such as reaching the end of the walk.) 
Therefore, using the probability mass function of the Poisson distribution and the moment generating
function of the geometric distribution, we see that for any $\alpha < 1/\lambda$, as $n$ goes to infinity, $\E[\alpha^{S_n}]$ converges to
\begin{align*}
  \sum_{k=0}^\infty \frac{\exp(-(1-\lambda)\Phi) ((1-\lambda)\Phi)^k}{k!} \parens[\Big]{\frac{(1-\lambda)\alpha}{1-\lambda \alpha}}^k &= \exp\parens[\Big]{-(1-\lambda)\Phi + (1-\lambda)\Phi \frac{(1-\lambda)\alpha}{1-\lambda \alpha}} \\
 &= \exp\parens[\bigg]{\Phi \cdot \parens[\bigg]{\frac{(1-\lambda)(\alpha-1)}{1-\alpha\lambda}}} \; ,
\end{align*}
as desired.

\section{Preliminaries}

For $p \ge 1$, let the $p$-norm of an $N$-dimensional vector $v$ be defined as 
\[ \|v\|_p= \parens[\Big]{\frac{|v_1|^p+|v_2|^p+\cdots + |v_N|^p}{N}}^{1/p}
\;.
\]  
Note that under this definition, $\|v\|_p \leq \|v\|_q$ if $p \leq q$.  
Additionally, we let the operator norm of a matrix $A \in \R^{N \times N}$ be defined as 
\[\|A\| = \max_{v: \|v\|_2 = 1} \|Av \|_2.\]
For a vector $v$, we let $\diag(v)$ be the diagonal matrix where $\diag(v)_{i, i} = v_i$.

We denote by $\mathbf{1}$ the all-$1$ vector, and by $J$ the matrix whose entries are all $1/N$.
Given a 
regular undirected graph $G = (V, E)$, let $A$ be its normalized adjacency matrix.  
We let $\lambda(G) = \|A-J\|$ (where $J$ is the matrix whose entries are all
 $1/|V|$) be the second largest eigenvalue in absolute value of $A$.

\section{Bounding monomials}\label{sec:submonom}

In this section we prove Lemma~\ref{lem:monomial}, bounding the expectation
of monomials in the $f_i(Y_i)$.  
We start with two simple claims. 

\begin{claim}\label{clm:jbetween}
For all $k \ge 1$ and matrices $R_1,\ldots,R_k \in \R^{N \times N}$,
\[
\|R_1 J R_2 J \cdots J R_k \mathbf{1} \|_1 \le \prod_{i=1}^k \|R_i \mathbf{1}\|_1 \; .
\]
\end{claim}
\begin{proof}
Notice that for any vector $v$, $Jv = \alpha \mathbf{1}$ where $\alpha$ is the average
of the coordinates of $v$ and hence satisfies $|\alpha| \le \|v\|_1$.  The claim 
then follows by induction. 
\end{proof}

\begin{claim}\label{clm:claimcombo}
For all $k \ge 1$, vectors $u_1, \ldots, u_{k} \in [0, 1]^N$, $U_i = \diag(u_i)$ for all $i$, 
and matrices $T_1, \ldots, T_{k-1} \in \R^{N \times N}$,
\[
\left\|U_1 T_1 U_2 T_2 \cdots T_{k-1} U_k \mathbf{1} \right\|_{1} 
\leq 
\sqrt{\|u_1\|_1\|u_{k}\|_1}\prod_{i=1}^{k-1} \|T_i\| \; .
\]
\end{claim}

\begin{proof}
The case $k=1$ is immediate since $U_1 \mathbf{1} = u_1$. 
For $k>1$, 
\begin{align*}
\left\|U_1 T_1 U_2 T_2 \cdots T_{k-1} U_k \mathbf{1} \right\|_{1} 
& \le
\|u_1\|_2 
\|T_1 U_2 T_2 \cdots T_{k-1} u_k \|_2 \\
& \le 
\|u_1\|_2 \|u_k \|_2 \prod_{i=1}^{k-1} \|T_i\|\; .
\end{align*}
where the first inequality is the Cauchy-Schwarz inequality, and the
second inequality is by the definition of operator norm and the observation
that $\|U_i\| = \|u_i\|_\infty \le 1$.
We complete the proof by noting that $\|u_i\|_2^2 \le \|u_i\|_1$
since the entries of $u_i$ are in $[0,1]$.
\end{proof}

We can now prove the main result of this section. 

\begin{lemma}\label{lem:monomial}
Let $G = (V, E)$ be a regular undirected graph and let $\lambda = \lambda(G)$.  
Let $f_1, \ldots, f_n: V \rightarrow [0, 1]$, and let $\mu_i = \E[f_i(v)]$.  For a random walk 
$(Y_1, \ldots, Y_n)$, let $Z_i = f_i(Y_i)$ for all $i$.
Then for all $k \ge 1$ and $w \in [n]^k$ such that 
$w_1 \leq w_2 \leq \cdots \leq w_k$, 
\begin{align}
\E[Z_{w_1}Z_{w_2}\cdots Z_{w_k}] 
&\leq
\sum_{s\in\{0, 1\}^{k-1}}
\sqrt{\mu_{w_1}\mu_{w_k}}\parens[\Bigg]{\prod_{i: s_i = 0}(1-\lambda^{w_{i+1}-w_i})\sqrt{\mu_{w_i}\mu_{w_{i+1}}}}
\parens[\Bigg]{\prod_{i: s_i = 1} \lambda^{w_{i+1}-w_i}} \nonumber \\
&=
\sqrt{\mu_{w_1}\mu_{w_k}}\prod_{i=1}^{k-1}
\parens[]
{{(1-\lambda^{w_{i+1}-w_{i}})\sqrt{\mu_{w_i}\mu_{w_{i+1}}}}+\lambda^{w_{i+1}-w_{i}}}. \label{eq:boundmonomials}
\end{align}
\end{lemma}

\begin{proof}
Let $d_i = w_{i+1}-w_i$ for all $i$, and let $A$ be the normalized adjacency matrix of $G$.  
Let $u_i$ be given by $(u_i)_v = f_{w_i}(v)$ and let $U_i = \diag(u_i)$.
Then
 \begin{align}\label{eq:monoboundfirst}
\E[Z_{w_1}Z_{w_2}\cdots Z_{w_k}] 
&=
\|U_{1}A^{d_1}U_{2}A^{d_2}\cdots A^{d_{k-1}}U_{k} \mathbf{1} \|_1.
\end{align}
Let $T_{i,0} = (1-\lambda^{d_i}) J$ and $T_{i,1} = A^{d_i} - (1-\lambda^{d_i}) J$
and notice that $\|T_{i,1}\|_2 = \lambda^{d_i}$.
Using the triangle inequality, we can bound the right-hand side of Eq.~\eqref{eq:monoboundfirst} from above by
\[
\sum_{s \in \{0,1\}^{k-1}}
\|U_{1}T_{1,s_1}U_{2}T_{2,s_2} \cdots T_{k-1,s_{k-1}} U_{k} \mathbf{1} \|_1.
\]
The main claim is that for each $s$, the term corresponding to $s$ satisfies
\begin{equation}\label{eq:mainclaiminmono}
\|U_{1}T_{1,s_1}U_{2}T_{2,s_2} \cdots T_{k-1,s_{k-1}} U_{k} \mathbf{1} \|_1
\le
\sqrt{\mu_{w_1}\mu_{w_k}}\parens[\Bigg]{\prod_{i: s_i = 0}(1-\lambda^{d_i})\sqrt{\mu_{w_i}\mu_{w_{i+1}}}}
\parens[\Bigg]{\prod_{i: s_i = 1} \lambda^{d_i}} \; .
\end{equation}
Notice that this claim immediately implies the lemma by summing over $s \in \{0,1\}^{k-1}$. 
To see why the claim is true, let $1 \le r_1 < \cdots < r_\ell \le k-1$ be the indices 
where $s$ is zero. Then by Claim~\ref{clm:jbetween}, the left-hand side of Eq.~\eqref{eq:mainclaiminmono}
is at most 
\begin{align*}
\parens[\Bigg]{\prod_{i: s_i = 0}(1-\lambda^{d_i})} \cdot
&\|U_{1}T_{1,1}U_{2} T_{2,1}\cdots T_{r_1-1,1} U_{r_1} \mathbf{1} \|_1 \cdot
\|U_{r_1+1}T_{r_1+1,1}U_{r_1+2} T_{r_1+2,1}\cdots T_{r_2-1,1} U_{r_2} \mathbf{1} \|_1 
\cdots \\
&\cdots
\|U_{r_\ell+1}T_{r_\ell+1,1}U_{r_\ell+2} T_{r_\ell+2,1}\cdots T_{k-1,1} U_{k} \mathbf{1} \|_1  \; .
\end{align*}
The claim now follows by applying Claim~\ref{clm:claimcombo}.
\end{proof}

When $\mu_1=\cdots=\mu_n=\mu$, the bound in Eq.~\eqref{eq:boundmonomials} simplifies to
\[
\E[Z_{w_1}Z_{w_2}\cdots Z_{w_k}]  
\le 
\mu \prod_{i=1}^{k-1}
\parens[]
{{(1-\lambda^{w_{i+1}-w_{i}})\mu}+\lambda^{w_{i+1}-w_{i}}}. 
\]
Observe that for the two-state Markov chain described in Section~\ref{sec:tightness},
for every $n \ge 1$ and every $1 \le w_1 \le \cdots \le w_k \le n$,
this inequality is actually an equality. Indeed, the left-hand side is the probability
that we are in the marked state at all the steps $w_1,\ldots,w_k$. 
The probability of being in the marked state at step $w_1$ is $\mu$ (as we are
in the stationary distribution); and the probability of being in the marked state at step
$w_{i+1}$ conditioned on being there at step $w_i$ is 
$(1-\lambda^{w_{i+1}-w_{i}})\mu+\lambda^{w_{i+1}-w_{i}}$.

This observation implies that the moment generating function $\E[\alpha^{S_n}]$
of an arbitrary graph and arbitrary $f_1,\ldots,f_n$ with all $\E[f_i]$
equal can be bounded by the moment generating function of the corresponding 
two-state Markov chain (as can be seen from the Taylor expansion;
see Eq.~\eqref{eq:taylormgf} below). This can be used to give an alternative 
(and perhaps more intuitive) proof of Theorem~\ref{thm:maindiff}. 
We do not include this proof here since it is
not clear how to extend it to the case of general $\mu_i$. 

\section{Proof of Theorem~\ref{thm:maindiff}}\label{sec:diff}

In this section we complete the proof of the main theorem using the bound in Lemma~\ref{lem:monomial}.
We start with the following easy corollary of Cauchy-Schwarz.  

\begin{claim}\label{claim:csapp}
Let $P \in \R[X_1, \ldots, X_n]$ be a multivariate polynomial with non-negative coefficients.
Then for $x_1, \ldots, x_n, y_1, \ldots, y_n \in \R$,
\[
P(x_1y_1, x_2y_2, \ldots, x_ny_n)
\leq
\max\{P(x_1^2, x_2^2, \ldots, x_n^2), P(y_1^2, y_2^2, \ldots, y_n^2)\}.
\]
\end{claim}
\begin{proof}
Let 
\[
P(X_1, \ldots, X_n) = \sum_{m \in \N^{n}} a_{m}X_1^{m_1} X_2^{m_2} \cdots X_n^{m_n}
\]
for some $a_m \ge 0$.
Then
\begin{align*}
P(x_1y_1, x_2y_2, \ldots, x_ny_n)
&=
\sum_{m \in \N^{n}} a_m(x_1 y_1)^{m_1}(x_2 y_2)^{m_2} \cdots (x_n y_n)^{m_n} \\
&=
\sum_{m \in \N^{n}}(\sqrt{a_m} x_1^{m_1}x_2^{m_2}\cdots x_n^{m_n})(\sqrt{a_m} y_1^{m_1}y_2^{m_2}\cdots y_n^{m_n}) \\
&\leq 
\parens[\Bigg]{\sum_{m \in \N^{n}}a_m x_1^{2m_1}x_2^{2m_2}\cdots x_n^{2m_n}}^{1/2}
\parens[\Bigg]{\sum_{m \in \N^{n}}a_m y_1^{2m_1}y_2^{2m_2}\cdots y_n^{2m_n}}^{1/2} \\
&\leq
\max\{P(x_1^2, x_2^2, \ldots, x_n^2), P(y_1^2, y_2^2, \ldots, y_n^2)\},
\end{align*}
where the first inequality follows from Cauchy-Schwarz.
\end{proof}

\begin{lemma}\label{lem:monomdiff}
Let $G = (V, E)$ be a regular undirected graph, let $\lambda = \lambda(G)$,
and let $f_1, \ldots, f_n: V \rightarrow [0, 1]$.  For a random walk 
$(Y_1, \ldots, Y_n)$, let $Z_i = f_i(Y_i)$ for all $i$, and let $\mu_i = \E[f_i(v)]$, $\Phi = \mu_1+\cdots+\mu_n$.  
For all $k \in [n]$, let $W_k \subseteq [n]^k$ be the set of all $w$ such that 
$w_1 < w_2 < \cdots < w_k$.  Then
\[
\E\bracks[\Bigg]{\sum_{w \in W_k} Z_{w_1}Z_{w_2}\cdots Z_{w_k}}
\leq
\sum_{i=0}^{k-1} \binom{k-1}{i}\frac{\Phi^{i+1}\lambda^{k-i-1}}{(i+1)!(1-\lambda)^{k-i-1}} \; .
\]
\end{lemma}
\begin{proof}
By Lemma~\ref{lem:monomial} and Claim~\ref{claim:csapp}, 
\begin{align*}
\E\bracks[\Bigg]{\sum_{w \in W_k} Z_{w_1}Z_{w_2}\cdots Z_{w_k}}
&\leq 
\sum_{w \in W_k} \sum_{s \in \{0, 1\}^{k-1}}
\sqrt{\mu_{w_1}\mu_{w_k}}\parens[\Bigg]{\prod_{i: s_i = 0}\sqrt{\mu_{w_i}\mu_{w_{i+1}}}}
\parens[\Bigg]{\prod_{i: s_i = 1} \lambda^{w_{i+1}-w_i}} \\
&
\begin{aligned}\hspace{1.25pt} \leq \max
\braces[\bigg]{
&\sum_{w \in W_k} \sum_{s \in \{0, 1\}^{k-1}}
\mu_{w_1}\prod_{i: s_i = 0}\mu_{w_{i+1}}
\prod_{i: s_i = 1} \lambda^{w_{i+1}-w_i}
, \\
&\sum_{w \in W_k} \sum_{s \in \{0, 1\}^{k-1}}
\mu_{w_k}\prod_{i: s_i = 0}\mu_{w_{i}}
\prod_{i: s_i = 1} \lambda^{w_{i+1}-w_i}}.
\end{aligned}
\end{align*}
We assume for the remainder of the proof that the second term is the maximum.
A similar proof holds under the assumption that the first term is the maximum.

We will show that for each $s \in \{0, 1\}^{k-1}$, 
\begin{equation}
\sum_{w \in W_k} 
\mu_{w_k}\parens[\Bigg]{\prod_{i: s_i = 0}\mu_{w_i}}
\parens[\Bigg]{\prod_{i: s_i = 1} \lambda^{w_{i+1}-w_i}}
\leq 
\frac{(\mu_1+\cdots+\mu_n)^{k-|s|}\lambda^{|s|}}{{(k-|s|)!}(1-\lambda)^{|s|}},
\label{eq:diffrearrange}
\end{equation}
where $|s|$ is the number of coordinates of $s$ equal to $1$.
This proves the lemma, as there are $\binom{k-1}{k-j-1}$ vectors $s \in \{0, 1\}^{k-1}$ such that $|s| = j$.  
From this point we fix $s$.

Let $w_{\bar{s}}$ be $w$ restricted to the coordinates $i$ such that $s_i = 0$ along with the $k$th coordinate.
Then
\begin{align*}
\sum_{w \in W_k} 
\mu_{w_k}\parens[\Bigg]{\prod_{i: s_i = 0}\mu_{w_i}}
\parens[\Bigg]{\prod_{i: s_i = 1} \lambda^{w_{i+1}-w_i}}
&=
\sum_{v \in W_{k-|s|}} \parens[\Bigg]{\prod_{j=1}^{k-|s|} \mu_{v_j}}
\parens[\Bigg]{\sum_{w:w_{\bar{s}} = v} \prod_{i:s_i = 1} \lambda^{w_{i+1}-w_i}} \\
&\leq
\sum_{v \in W_{k-|s|}} \parens[\Bigg]{\prod_{j=1}^{k-|s|} \mu_{v_j}}
 \parens[\bigg]{\frac{\lambda}{1-\lambda}}^{|s|} \; .
\end{align*}
where the last inequality uses the observation that the function that maps any $w\in W_{k-|s|}$
with $w_{\bar{s}}=v$ to the sequence of positive values $(w_{i+1}-w_i)_{i:s_i=1}$ 
is an injective function.
Finally, Eq.~\eqref{eq:diffrearrange} follows by noting that
\[
\sum_{v \in \binom{[n]}{k-|s|}} \parens[\Bigg]{\prod_{j=1}^{k-|s|} \mu_{v_j}}
\leq
\frac{(\mu_1+\cdots+\mu_n)^{k-|s|}}{(k-|s|)!}.
\]
\end{proof}

The following lemma gives an upper bound on the moments of $S_n$.

\begin{lemma}\label{lem:momentbounddifferent}
Let $G = (V, E)$ be a regular undirected graph, let $\lambda = \lambda(G)$,  
and let $f_1, \ldots, f_n: V \rightarrow [0, 1]$.  
For a random walk 
$(Y_1, \ldots, Y_n)$, let $Z_i = f_i(Y_i)$ for all $i$, and let $S_n = Z_1+\cdots+Z_n$,
and let $\mu_i = \E[Z_i]$, $\Phi = \mu_1+\cdots+\mu_n$.
Then for all positive integers $q$,
\[
\E[S_n^q]
\leq
\sum_{k=1}^{q}\stirling{q}{k}k!\sum_{i=0}^{k-1} \binom{k-1}{i}
\frac{\Phi^{i+1}\lambda^{k-i-1}}{(i+1)!(1-\lambda)^{k-i-1}}
\]
where $\stirling{}{}$ denotes the Stirling number of the second kind.
\end{lemma}
\begin{proof}
Consider the subset $D_k \subseteq [n]^q$ of vectors with exactly $k$ distinct coordinates.
Note that
\[\E[S_n^q] = \sum_{w \in [n]^q} \E\bracks[\Bigg]{\prod_{j=1}^q Z_{w_j}} = \sum_{k=1}^q \sum_{w \in D_k} \E\bracks[\Bigg]{\prod_{j=1}^q Z_{w_j}}.\] 
We will upper bound each term on the right-hand side separately.

Fix a $k$, and let $W_k \subseteq [n]^k$ be the set of vectors $w$ so that $w_1 < w_2 < \cdots < w_k$.
Let $\psi: D_k \rightarrow W_k$ be the function 
mapping each $w \in D_k$ to the vector whose coordinates
are exactly those in $w$ in sorted order and without repetition.  Then because $Z_i \in [0, 1]$ for all $i$,
\begin{equation}\label{eq:removerepeats}
\sum_{w \in D_k} \E\bracks[\Bigg]{\prod_{j=1}^q Z_{w_j}}
\leq
\sum_{w \in D_k} \E\bracks[\Bigg]{\prod_{j=1}^k Z_{\psi(w)_j}}.
\end{equation}
Moreover, for all $w \in W_k$ we have $|\psi^{-1}(w)| = \stirling{q}{k}k!$ (as this is the number of ways to partition $q$ labeled balls into $k$ nonempty labeled boxes), and thus Eq.~\eqref{eq:removerepeats} is equal to
\[
\stirling{q}{k}k!\sum_{w \in W_k}\E\bracks[\Bigg]{\prod_{j=1}^k Z_{w_j}}.
\]
The lemma then follows from Lemma~\ref{lem:monomdiff}.
\end{proof}

Finally we can insert the upper bounds from Lemma~\ref{lem:momentbounddifferent} in the
Taylor expansion of $\alpha^{S_n}$ to prove Theorem~\ref{thm:maindiff}.

\begin{proof}[Proof of Theorem~\ref{thm:maindiff}]
By Lemma~\ref{lem:momentbounddifferent}, 
\begin{align}
\E[\alpha^{S_n}] 
&= 
\sum_{q=0}^\infty \frac{\log(\alpha)^q \E[S_n^q]}{q!} \label{eq:taylormgf} \\
&\le
1+\sum_{q=1}^{\infty}\frac{\log(\alpha)^q}{q!} \sum_{k=1}^{q}\stirling{q}{k}k!
\sum_{i=1}^{k} \binom{k-1}{i-1}\frac{\Phi^{i}\lambda^{k-i}}{i!(1-\lambda)^{k-i}}. \nonumber
\end{align}
Rearranging the sums yields 
\begin{equation}
1+\sum_{i=1}^{\infty}\frac{\Phi^{i}}{i!}\sum_{k = i}^{\infty} \binom{k-1}{i-1}\frac{\lambda^{k-i}}{(1-\lambda)^{k-i}}\sum_{q=k}^{\infty} \stirling{q}{k}\frac{\log(\alpha)^q k!}{q!}.
\label{eq:rearrange}\end{equation}
Using the following identity~\cite[Eq.~1.94(b)]{S12}, 
\[\sum_{q=k}^{\infty} \stirling{q}{k}\frac{\log(\alpha)^q}{q!} = \frac{(\alpha-1)^k}{k!},\] 
(which can be seen by writing $\alpha - 1 = e^{\log(\alpha)} - 1$ as $\log(\alpha) + \frac{1}{2!} \log(\alpha)^2 + \frac{1}{3!} \log(\alpha)^3 + \cdots$)
we can rewrite Eq.~\eqref{eq:rearrange} as 
\begin{align} 
1+\sum_{i=1}^{\infty}\frac{\Phi^{i}}{i!} (\alpha-1)^{i} \sum_{k = i}^{\infty} \binom{k-1}{i-1}\frac{\lambda^{k-i}}{(1-\lambda)^{k-i}}{(\alpha-1)^{k-i}}  \; .
\label{eq:identity}
\end{align}
Using the following identity for $0 \le x < 1$,
\[\sum_{j=i}^{\infty} \binom{j-1}{i-1}x^{j-i} = {(1-x)^{-i}},\] 
(which follows from differentiating $\sum_{j=0}^{\infty} x^j = (1-x)^{-1}$ a total of $i-1$ times) 
we can rewrite Eq.~\eqref{eq:identity} as 
\begin{align*}
1+\sum_{i=1}^{\infty}\frac{\Phi^{i}({\alpha-1})^{i}}{i!}
\parens[\bigg]{1-\frac{\lambda(\alpha-1)}{1-\lambda}}^{-i}  
 =
\exp\parens[\bigg]{\Phi\cdot \parens[\bigg]{\frac{(1-\lambda)(\alpha-1)}{1-\alpha\lambda}}} \; .
\end{align*}  
\end{proof}

\paragraph{Acknowledgements}

We thank Assaf Naor for many useful discussions, and Michael Forbes
for referring us to prior work.
We also thank the Oberwolfach Research Institute for Mathematics and the organizers of the 
``Complexity Theory'' workshop there in November 2015 where this work was initiated.

\bibliographystyle{alphaabbrvprelim}
\bibliography{expsamp}

\end{document}